\documentclass[12pt]{amsart}
\usepackage{a4ams}

\usepackage{amsfonts}
\usepackage{amssymb}
\usepackage{amsmath}
\usepackage[arrow,matrix,curve]{xy}

\newcommand{\C}{{\mathbb C}}

\newcommand{\A}{{\mathbb A}}
\newcommand{\T}{{\mathbb T}}
\newcommand{\F}{{\mathbb F}}
\newcommand{\I}{{\mathbb I}}
\newcommand{\B}{{\mathbb B}}
\newcommand{\J}{{\mathbb J}}

\newcommand{\G}{{\mathbb G}}

\newtheorem{theorem}{Theorem}[section]
\newtheorem{lemma}[theorem]{Lemma}

\newtheorem{corollary}[theorem]{Corollary}

\newtheorem{definition}[theorem]{Definition}

\def\cal{\mathcal}

\newcommand{\set}[2]{\{ \, #1 \,\, | \, \, #2 \, \} }

\newcommand{\calh}[0]{{\cal H}}

\newcommand{\cala}[0]{{\cal A}}

\newcommand{\calk}[0]{{\cal K}}

\newcommand{\calp}[0]{{\cal P}}
\newcommand{\cali}[0]{{\cal I}}

\newcommand{\calm}[0]{{\cal M}}

\newcommand{\calj}[0]{{\cal J}}

\newcommand{\lin}[0]{ \mbox{\rm{lin}}}

\newcommand{\bal}[0]{ {\rm bal} }

\newcommand{\aut}[0]{ \mbox{\rm{Aut}}}

\newcommand{\alg}[0]{ {\rm Alg} }

\usepackage[normalem]{ulem}
\usepackage{color}

\begin{document}

\title{On certain properties of Cuntz--Krieger type algebras}
\author[B. Burgstaller and D. G. Evans]{B. Burgstaller and D. G. Evans}
\address{Mathematisches Institut, Einsteinstra\ss e 62, 48149 M\"unster, Germany
and
Institute of Mathematics and Physics,
Aberystwyth University,
Ceredigion SY23 3BZ,
Wales UK
}
\email{bernhardburgstaller@yahoo.de}
\email{dfe@aber.ac.uk }
\thanks{Both authors were supported by the EU-project Quantum Space and Noncommutative
Geometry, HPRN-CT-2002-00280.}
\subjclass{46L05, 46L55}
\keywords{Cuntz--Krieger, semigraph algebra, ideal, purely infinite, crossed product, nuclear}

\begin{abstract}
The note presents a further study of the class of Cuntz--Krieger type algebras.
A necessary and sufficient condition is identified
that ensures that the algebra is purely infinite,
the ideal structure is studied,
and nuclearity
is proved by presenting the algebra as a crossed product of an AF-algebra by an abelian group.
The results are applied to examples of Cuntz--Krieger type algebras,
such as higher rank semigraph $C^*$-algebras
and higher rank Exel-Laca algebras.
\end{abstract}

\maketitle

\section{Introduction}

Based on the work of Cuntz and Krieger in \cite{cuntzkrieger},
in \cite{burgi2} the first named author considered a class of so-called Cuntz--Krieger type algebras relying on a flexible
generators and relations approach. This class, which is recalled in Section \ref{section2}, includes (aperiodic) Cuntz--Krieger algebras
\cite{cuntzkrieger}, higher rank Exel--Laca algebras \cite{burgiHREL},
(aperiodic) higher rank graph $C^*$-algebras \cite{kumjianpask,raeburnsimsyeendFinitelyAligned},
(aperiodic) ultragraph algebras \cite{tomforde} and (cancelling) higher rank semigraph $C^*$-algebras
\cite{burgiSemigraphI}.

The aim of this note is to analyse these algebras further.
In Section \ref{section3} we show that a Cuntz--Krieger type algebra is purely infinite if and only if
the projections of its core are infinite, see Theorem \ref{piprop}. Corresponding quite tractable adapted reformulations
for higher rank semigraph $C^*$-algebras and higher rank Exel--Laca algebras
are stated in
Corollaries \ref{corollaryPurelyInfiniteSemigraph} and \ref{corollaryPurelyInfiniteHREL}, respectively.
In Section \ref{section4} we study the ideal structure of Cuntz--Krieger type algebras.
There is an injection of certain ideals of the core to the ideals of the Cuntz--Krieger type algebra,
see Theorem \ref{injectlattice}. If these certain ideals are all cancelling (Definitions \ref{DefCancelableFirst} and
\ref{DefCancelable})
then this injection is even a lattice isomorphism, see
Theorem \ref{idealspropertyC}, Corollary \ref{corollaryIdealsIso}, Theorem \ref{idealspropertyCnumber2}
and Corollary \ref{corollaryIdealsIso2}.
We give reformulations of such an isomorphism for higher rank semigraph algebras in Corollaries
\ref{corollarySemigraphIsomorphPhiB} and \ref{corollarySemigraphIsomorphPhiB2}.
In Section \ref{section5} we present the stabilised Cuntz--Krieger type algebras as crossed products of AF-algebras
by abelian groups, see Theorem \ref{takaidualityapplication}. Hence Cuntz--Krieger type algebras are nuclear.





\section{Cuntz--Krieger type algebras} \label{section2}

We briefly recall the basic definitions and facts of the class of Cuntz--Krieger type algebras
introduced in
\cite{burgi2} and slightly extended in \cite{burgiNoteUniqueness}.  

Assume that we are given an alphabet $\cala$, the free nonunital $*$-algebra $\F$
generated by $\cala$, a two-sided self-adjoint ideal $\I$ of $\F$,
and a closed subgroup $H$ of $\T^\cala$ ($\T$ denotes the circle).
We are interested in the quotient $*$-algebra $\F/\I$ and its universal $C^*$-algebra
$C^*(\F/\I)$.
Denote the set of words of $\F/\I$ by $W = \{ a_1 \ldots a_n \in \F/\I|\, a_i \in \cala \cup \cala^*\}$.
(We will always write $x$ rather than $x+\I$ in the quotient $\F/\I$ for elements $x \in \F$ if there is no danger of confusion.)
An element $x$ of a $*$-algebra is called a partial isometry if $x x^* x=x$, and a projection if $x^2 =x^* = x$.

We introduce the following properties (A), (B) and (C') for the system $(\cala,\F,\I,H)$.

\begin{itemize}
\item[(A)]
There exists a gauge action $t:H \longrightarrow \aut(\F/\I)$ 
determined by $t_\lambda (a) = \lambda_a a$ for all $a \in \cala$ and $\lambda = (\lambda_b)_{b \in \cala} \in H$.
\end{itemize}

Denote by $(\hat H,+,0)$ the character group of $(H,\cdot,1)$; note that we write the group operation of $\hat H$ additively.
The gauge action $t$ induces a so-called balance function $\bal: W \backslash \{0\} \longrightarrow \hat H$
from the nonzero words of $\F/\I$ to the character group $\hat H$ determined by
$\bal(a)((\lambda_b)_{b \in \cala}) = \lambda_a	\in \T$, $\bal(xy)= \bal(x)+\bal(y)$ and $\bal(x^*) = -\bal(x)$, where $a \in \cala$,
$(\lambda_b)_{b \in \cala} \in H \subseteq \T^\cala$ and $x,y \in W$ (see \cite[Lemma 3.1]{burgi2}).

Define $\A$ to be the linear span in $\F/\I$ of all words $x \in W \backslash \{0\}$ satisfying $\bal(x)=0$.
Actually, $\A$ is a $*$-algebra.
Words $x$ with balance $\bal(x)=0$ are called zero-balanced.
Write $W_n$ for the set of words with balance $n \in \hat H$.
Since every element of $\F/\I$ is expressable as a linear combination of words, we may write
$\F/\I = \sum_{n \in \hat H} \lin(W_n)$.
Note, however, that this sum might not be a direct sum.

\begin{itemize}
\item[(B)]
$\A$ is locally matricial, that is,
for all $x_1,\ldots,x_n \in \A$ there exists a finite dimensional
$C^*$-subalgebra $A$ of $\A$ such that $x_1,\ldots,x_n \in A$.

\item[(C')]
For every nonzero-balanced word $x \in W\backslash W_0$ and every nonzero projection
$e \in \A$ there exists a nonzero projection $p \le e$ in $\A$ such that
$pxp=0$.
\end{itemize}

\begin{definition}
{\rm
A system $(\cala,\F,\I,H)$ is called a Cuntz--Krieger type system, or $\F/\I$ is called a Cuntz--Krieger type $*$-algebra,
if (A), (B) and (C') are satisfied and
there exists a $C^*$-representation $\pi:\F/\I \longrightarrow A$ which is injective on $\A$.
}
\end{definition}

Throughout assume that $(\cala,\F,\I,H)$ is a Cuntz--Krieger type system if nothing else is said.
There exists a universal enveloping $C^*$-algebra $C^*(\F/\I)$ for $\F/\I$,
and clearly the universal representation $\zeta : \F/\I \longrightarrow C^*(\F/\I)$ is injective on $\A$.
The enveloping $C^*$-algebra $C^*(\F/\I)$ is called the Cuntz--Krieger type algebra
associated to $(\cala,\F,\I,H)$.
A $*$-homomorphism $\F/\I \rightarrow A$ into a $C^*$-algebra $A$ is called
a {\em $C^*$-representation of $\F/\I$}, and {\em $\A$-faithful} if it is faithful on $\A$.
We remark that for a system $(\cala,H,\F,H)$ satisfying (A), (B) and (C'), an $\A$-faithful representation of $\F/\I$ into a $C^*$-algebra exists automatically if the word set $W$ consists of partial
isometries, see \cite[Theorem 3.1]{burgiCancelling}.

We have the following Cuntz--Krieger uniqueness theorem.

\begin{theorem} \label{CKuniqueness}
If $\pi: \F/\I \longrightarrow A$ is an $\A$-faithful representation into a $C^*$-algebra $A$ with dense image in $A$
then
$A$ is canonically isomorphic to $C^*(\F/\I)$ via $\pi(x) \mapsto \zeta(x)$, so $\pi$ is essentially the universal map $\zeta$ (see
\cite[Theorem 2.1]{burgiNoteUniqueness}).
\end{theorem}

The next lemma states that we usually may assume without loss of generality that $\zeta$ is injective.
We then usually avoid notating $\zeta$ and regard $\F/\I$ as a subset of $C^*(\F/\I)$.

\begin{lemma}  \label{lemmaInjectZeta}
We may assume without loss of generality that the universal representation $\zeta:\F/\I \longrightarrow C^*(\F/\I)$
is injective by dividing out the kernel of $\zeta$. The new quotient $\F/\I$ is a Cuntz--Krieger $*$-algebra
again ($\cala,\F$ and $H$ remain unchanged). $\A$ remains unchanged under this modification.
\end{lemma}

In a previous preprint of this note we proved the last lemma and the next lemma. However,
we have reproved and published them already now in \cite{burgiCancelling} in \cite[Propositions 2 and 4]{burgiCancelling}.
The setting in \cite{burgiCancelling} generalises the setting of this note by allowing the image of the balance function,
here the commutative group $\hat H$, to be a non-commutative group.
Say that a $*$-algebra $X$ satisfies the $C^*$-property if for every $x \in X$, $x x^*=0$ implies $x=0$.

\begin{lemma}
$\zeta$ is injective if and only if $\F/\I$ satisfies the $C^*$-property.
The kernel of $\zeta$ is the ideal generated by $\{x \in \F/\I| \, x x^*=0\}$.
\end{lemma}

\begin{lemma} \label{lemmacondexpect}
There exists a conditional expectation
$F:C^*(\F/\I) \longrightarrow C^*(\A) \subseteq C^*(\F/\I)$ determined by
$F (\zeta(w)) = 1_{\{\bal(w)= 0\}} \zeta(w)$
for
words $w \in W$ (see \cite[Proposition 2]{burgiCancelling}).
\end{lemma}

\section{Pure Infiniteness} \label{section3}

In this section we analyse the pure infiniteness of a Cuntz--Krieger type algebra
$C^*(\F/\I)$.

Recall that a projection $p$ in a $C^*$-algebra $A$ is called infinite if it is the source
projection $s^*s$ of a partial isometry $s$ in $A$ with range projection $s s^*$ being smaller than $p$.
Recall the following simple lemma.

\begin{lemma} \label{pilemma}
If a projection is infinite then any other projection which is bigger in Murray--von Neumann order is also infinite. 
\end{lemma}

\begin{theorem} \label{piprop}
A Cuntz--Krieger type algebra
$C^*(\F/\I)$ is purely infinite if and only if
every nonzero projection of $\A$ is infinite in $C^*(\F/\I)$.
\end{theorem}

\begin{proof}
We assume that $\zeta$ is injective (Lemma \ref{lemmaInjectZeta}).
Define $A=C^*(\F/\I)$.
Assume that $A$ is purely infinite. Then for any nonzero
projection $e \in \A$ the hereditary $C^*$-algebra $e A e$
contains some infinite projection $p$.
Since $p \le e$, $e$ is infinite in $A$ by Lemma \ref{pilemma}.

To prove the other direction, assume that every nonzero projection in $\A$ is infinite in $A$.
It is proved in
Theorem 2.1 of \cite{burgiNoteUniqueness} that there exists a larger
Cuntz--Krieger type system $S=(\cala \times \calp, \G,\J,H \times \{1\})$
than $(\cala,\F,\I,H)$ such that
$\G/\J \cong \F/\I \otimes \F'/\I'$, where $\F'/\I'$ is a commutative unital locally matricial algebra,
and the system $S$ satisfies property (C) of \cite{burgi2}.
(The accurate assertion of (C) is here unimportant, as we will only need it to apply a lemma of \cite{burgi2}).
If we can show that $C^*(\G/\J) \cong C^*(\F/\I) \otimes C^*(\F'/\I')$ is purely infinite,
then it is not difficult to check that $C^*(\F/\I)$ is also purely infinite.
(The following fact holds in general: If $A \otimes D$ is purely infinite for two $C^*$-algebras
$A$ and $D$ where $D$ is unital and commutative, then
$A$ is purely infinite.)

That is why we may assume without loss of generality in what follows that the system $(\cala,\F,\I,H)$
satisfies property (C) of \cite{burgi2}.
To show that $A=C^*(\F/\I)$ is purely infinite, we imitate the proof of \cite[Proposition 5.11]{robertsonsteger2}.
Let $h$ be a nonzero positive element of $A$.
We have to show that $\overline{h A h}$ contains an infinite projection.
Let $\varepsilon > 0$, and choose $y \ge 0$ in $\F/\I$ such that $\|y-h^2\| \le \varepsilon$.

By \cite[Lemma 2.6]{burgi2} (applied to $\pi=\zeta$)
we are provided with a faithful expectation
$F: A \rightarrow C^*({\A})$
such that for every representation
$y= \sum_{\gamma \in \hat H} y_\gamma$ (where $y_\gamma \in \lin (W_\gamma)$)
there exists a projection $Q \in \A$ satisfying
$Q y Q = Q y_1 Q \in \A$ and $\|F y\| = \|Q y Q \|$.

We may assume without loss of generality that $\|F h^2 \| = 1$.
We have
$$\|F y \| \ge \|F h^2\| - \varepsilon = 1 - \varepsilon.$$
Let $Q y Q \in \calm$ for some finite dimensional $C^*$-algebra $\calm \subseteq \A$.
We choose a system of generating matrix units for $\calm$
such that the positive element $Q y Q$ has diagonal form in $\calm=M_{k_1} \oplus \ldots \oplus M_{k_d}$.
By projecting on the largest diagonal entry, we can choose a positive operator $R_1 \in \calm$
such that $P=R_1 Q y Q R_1$ is a projection and $\|R_1\| \le (1 - \varepsilon)^{-1/2}$.
By hypothesis $P \in \A$ is an infinite projection.

It follows that $\|R_1 Q h^2 Q R_1 - P\| \le \|R_1^2\| \|Q\|^2 \|y-h^2\| \le \varepsilon /(1-\varepsilon)$.
By functional calculus one obtains $R_2 \in A_+$, so that $R_2 R_1 Q h^2 Q R_1 R_2$
is a projection and
$$\|R_2 R_1 Q h^2 Q R_1 R_2 - P\| \le 2 \varepsilon/(1-\varepsilon).$$
For small $\varepsilon$ one can then find an element $R_3$ in $A$ such that
$$R_3 R_2 R_1 Q h^2 Q R_1 R_2 R_3^* = P.$$
Let $R=R_3 R_2 R_1 Q$, so that $R h^2 R^*=P$. Consequently, $R h$ is a partial isometry,
whose initial projection $h R^* R h$ is a projection in $h A h$ and whose
final projection is $P$.
Moreover, if $V$ is a partial isometry in $A$ such that $V^* V = P$ and
$V V^* < P$, then $(h R^*)V(R h)$ is a partial isometry in $h A h$ with initial
projection $h R^* R h$ and final projection strictly less
than $h R^* R h$.
\end{proof}

We shall now apply the last theorem to cancelling higher rank
semigraph algebras \cite{burgiSemigraphI}, which are special Cuntz--Krieger
type $*$-algebras.

\begin{corollary}   \label{corollaryPurelyInfiniteSemigraph}
A cancelling semigraph $C^*$-algebra $C^*(\F/\I)$ (see \cite[Definitions 5.1 and 7.2]{burgiSemigraphI}) is purely infinite if and only if every standard projection (see \cite[Definition 5.14]{burgiSemigraphI}) is infinite in $C^*(\F/\I)$.
\end{corollary}

\begin{proof}
Cancelling semigraph algebras are algebras of amenable Cuntz--Krieger systems \cite{burgiCancelling} (this follows from the discussion in
\cite[Section 7]{burgiSemigraphI}), which again are Cuntz--Krieger type $*$-algebras (since the image of the balance map, $\hat H$, is an
abelian group). So we can apply Theorem \ref{piprop}.
We just need to recall that
by \cite[Corollary 6.4]{burgiSemigraphI} every nonzero projection in $\A$ is larger or equal than a standard projection in Murray--von Neumann order,
and so is infinite by Lemma \ref{pilemma} if every standard projection is infinite.
\end{proof}

The next corollary concerns
higher rank Exel--Laca algebras \cite{burgiHREL}, which are special Cuntz--Krieger type algebras.

\begin{corollary}   \label{corollaryPurelyInfiniteHREL}
Let $C^*(\F/\I)$ be a higher rank Exel--Laca algebra
\cite{burgiHREL}.
Then
$C^*(\F/\I)$
is purely infinite if and only if every nonzero projection of the form
$P_{a_1} \ldots P_{a_n}$ ($a_i \in \cala$, $P_a = a a^*$)
is infinite in $C^*(\F/\I$).
\end{corollary}

\begin{proof}
By \cite[Proposition 3.3]{burgiNoteUniqueness}, \cite[Corollary 4.14]{burgiHREL}
and \cite[Lemma 4.5]{burgiHREL} every projection $p \in \A$ allows the following estimate in
Murray--von Neumann order:
$$p \succsim x x^* \succsim x^* x = Q_{a_1} \ldots Q_{a_n} \ge P_{b_1} \ldots P_{b_n} \neq 0$$
for some word $x$ in the letters of the alphabet $\cala$, and some letters $a_i,b_i \in \cala$.
Hence, the claim follows from Lemma \ref{pilemma} and Theorem \ref{piprop}.
\end{proof}

\section{Ideal structure} \label{section4}

In this section we investigate the ideal structure of a Cuntz--Krieger
type algebra $C^*(\F/\I)$.
We assume that $\zeta$ is injective (Lemma \ref{lemmaInjectZeta}).

Write $\Sigma$ for the set of two-sided self-adjoint ideals in $\F/\I$. Denote by
$\cali$ the set of closed two-sided ideals in $C^*(\F/\I)$.
Suppose that $\B$ is a $*$-subalgebra of $\A$.
Write $\Sigma^{\B}$ for the set of self-adjoint two-sided ideals in $\B$.
Define
$$
\Sigma_\B = \set{J \cap \B \in \Sigma^\B}{\, J \in \Sigma}.
$$
For a subset $X$ of $\F/\I$, define
$\Sigma(X) \in \Sigma$ to be the two-sided self-adjoint ideal in $\F/\I$ generated by $X$,
and
$\cali (X) \in \cali$
the closed two-sided ideal in $C^* (\F/\I)$ generated by $X$.
Denote by $q_X:\F/\I \longrightarrow (\F/\I)/\Sigma(X)$ the quotient map.

\begin{lemma} \label{someidealminilemma}
For all $J \in \Sigma$ one has
$J \cap \B = (\Sigma( J \cap \B )) \cap \B$.
\end{lemma}

\begin{proof}
$J \cap \B \subseteq J \cap \B \cap \B \subseteq (\Sigma(J \cap \B)) \cap \B \subseteq \Sigma(J) \cap \B = J \cap \B$.
\end{proof}

\begin{lemma} \label{lemmaIdemIdeal}
We have $\Sigma_\B = \set{J \cap \B \in \Sigma^{\B}}{ J \in \Sigma, \, J= \Sigma(J \cap \B)}$.
\end{lemma}

\begin{proof}
Given $J \in \Sigma$, consider $I=\Sigma(J \cap \B)$. By Lemma \ref{lemmaIdemIdeal}
we have $I=\Sigma(I \cap \B)$ and $J \cap \B = I \cap \B$, which proves the claim.
\end{proof}

\begin{lemma} \label{lemmaIdemIdeal2}
We have $\Sigma_\B = \set{I \in \Sigma^{\B}}{ \Sigma(I) \cap \B = I}$.
\end{lemma}

\begin{proof}
Given $I \in \Sigma_\B$,
we have $I=J \cap \B$ for some ideal $J \in \Sigma$.
By Lemma \ref{someidealminilemma} we obtain $\Sigma(I) \cap \B = I$. The reverse implication is obvious.
\end{proof}

\begin{lemma}   \label{lemmaLatticaA}
We have
\begin{equation}   \label{equ213}
\Sigma_\A \,=\, \set{ I \in \Sigma^\A }{ \forall x,y \in W: \bal(x)+\bal(y)=0 \,\,\Longrightarrow
\,\, x I y \subseteq I }.
\end{equation}
Hence $\Sigma_\A$ is closed under the lattice operation $I+J$.
\end{lemma}

\begin{proof}
Write $\calj$ for the righthanded set of (\ref{equ213}).
Consider $I \in \Sigma_\A$ and write it as $I=J \cap \A$ for some $J \in \Sigma$.
If $i \in I$ and $x,y \in W$ with $\bal(x)+ \bal(y)=0$ then
$xiy \in \A \cap J$. This shows that $\Sigma_\A \subseteq \calj$.

To prove $\calj \subseteq \Sigma_\A$, consider $I \in \calj$.
Since $I \subseteq \A$,
$I \subseteq \Sigma(I) \cap \A$.
For the reverse inclusion
consider $z \in \Sigma(I) \cap \A$.
We may write $z= \sum \alpha_k x_k i_k y_k$
for some scalars $\alpha_k \in \C$, some $i_k \in I$, and some (possibly empty) words $x_k , y_k \in W$.
We have $F(z)=z$ for the conditional expectation $F$ of Lemma \ref{lemmacondexpect} as
$z \in \A$.
Hence $z= \sum \beta_k x_k i_k y_k$ for some scalars $\beta_k \in \C$ such that
$\beta_k = 0$ if $\bal(x_k)+\bal( y_k) \neq 0$.
This shows that $z \in I$ as $I \in \calj$.
We have proved that $I= \Sigma(I) \cap \A$, which is in $\Sigma_\A$.
\end{proof}

In the next lemma we state a result of Bratteli \cite{bratteli},
now for not necessarily separable AF-algebras.
We skip the proof which just consists of a slight adaption of Bratteli's proof.

\begin{lemma} \label{bratellilemma}
Let $A$ be a locally matricial algebra and $\overline A$ its $C^*$-algebraic norm closure.
There is a bijection $\gamma$ between the family of self-adjoint two-sided ideals in $A$
and the family of closed two-sided ideals in $\overline A$ through
$\gamma(I) = \overline I$ and $\gamma^{-1}(I) = I \cap A$.
%
\end{lemma}

\begin{theorem} \label{injectlattice}
Every $*$-subalgebra $\B$ of $\A$ induces an injective map
$\Phi_\B: \Sigma_\B
\longrightarrow \cali$ given by $\Phi_\B(I) =
\cali( I)$ for $I \in \Sigma_\B$.
The inverse map is determined by $\Phi^{-1}_\B(D) = D \cap \B$ for $D \in \cali$.
For all $I,J \in \Sigma_\B$
we have
\begin{eqnarray*}
\Phi_\B (I+J) &=& \Phi_\B (I) + \Phi_\B(J) \qquad \mbox{if } I+J \in \Sigma_\B,\\
\Phi_\B (I \cap J) &=& \Phi_\B (I) \cap \Phi_\B(J) \qquad \mbox{if }\Phi_\B (I) \cap \Phi_\B(J) \in \Phi_\B(\Sigma_\B).
\end{eqnarray*}
\end{theorem}

\begin{proof}
{\em Step 1.}
At first we are going to check injectivity of $\Phi_\A$.
Let $I \in \Sigma_\A$, and put $D= \cali(I)$.
Then $\overline I \subseteq \overline{D \cap \A}$ (norm-closures in $C^*(\F/\I)$).
To prove the reverse inclusion $\overline{D \cap \A} \subseteq \overline I$,
suppose that $x \in D \cap \A$.
Let $\varepsilon >0$. Since $D = \overline{\Sigma(I)}$,
there is some $y \in \Sigma(I)$ such that $\|x -y\| \le \varepsilon$.
Let $F$ be the conditional expectation of Lemma \ref{lemmacondexpect}.
Since $F x =x$, we have
$$\|x - Fy\| = \|Fx - F y\| \le \|x-y\| \le \varepsilon.$$
Choose for $y$ a representation $y=\sum \alpha_i a_i x_i b_i$ for some scalars $\alpha_i \in \C$,
some (possibly empty) words $a_i,b_i \in W$,
and some elements $x_i \in J$.
Since $\bal(x_i) = 0$, either $F(a_i x_i b_i) = a_i x_i b_i$ or $F(a_i x_i b_i)= 0$.
Hence $Fy = \sum \beta_i a_i x_i b_i \in \A$ for some scalars $\beta_i \in \C$,
and consequently
$F y \in \Sigma(I) \cap \A = I$
by Lemma \ref{lemmaIdemIdeal2}.
Since $\varepsilon >0$ was arbitrary, $x \in \overline I$.

We have proved that $\overline I = \overline{D \cap \A}$, and so $I= D \cap \A$
by Lemma \ref{bratellilemma}.
Hence $\Phi_\A^{-1} \Phi_\A(I) = I$ if we set $\Phi_\A^{-1}(D) = D \cap \A$.
Hence $\Phi_\A$ is injective.

{\em Step 2.}
In this step we will show that $\Phi_\B$ injective.
Define
$\mu:\Sigma_\B \rightarrow \Sigma_\A$ by
$\mu(I)= \Sigma(I) \cap \A$.
The map $\mu$ is injective as $\mu^{-1}(J) = J\cap \B$ is an inverse for $\mu$ by
Lemma \ref{lemmaIdemIdeal2}.
The identity
$$
\Phi_\A (\mu(I)) = \Phi_\A (\Sigma(I) \cap \A )
=
\overline{\Sigma (\Sigma(I) \cap \A ) }
=\overline{\Sigma(I) }
= \Phi_\B(I)
$$
shows that $\Phi_\B = \Phi_\A \mu$, and so $\Phi_\B$ is injective
by the proved injectivity of $\Phi_\A$.
To prove the formula for $\Phi_\B^{-1}$ we note that
$$\Phi_\B^{-1}(D) = \mu^{-1} \Phi_\A^{-1} (D) = (D \cap \A) \cap \B = D \cap \B.$$

{\em Step 3.}
To prove the lattice rules for $\Phi_\B$
we consider
$I_1,I_2 \in \Sigma_\B$ and set $D_1= \Phi_\B(I_1), D_2 =\Phi_\B(I_2)$.
If $D_1 \cap D_2 \in \Phi_\B(\Sigma_\B)$ then
$$\Phi^{-1}_\B(D_1 \cap D_2) = \Phi^{-1}_\B(D_1) \cap \Phi^{-1}_\B(D_2) = I_1 \cap I_2,$$
which shows $D_1 \cap D_2 = \Phi_\B(I_1 \cap I_2)$.
If $I_1 + I_2 \in \Sigma_\B$ then
$$
\Phi_\B (I_1 + I_2) = \overline{\Sigma(I_1 + I_2)}
= \overline{ \Sigma (D_1 + D_2 )}
= D_1 + D_2.$$
\end{proof}

We need a lemma which is often used in the theory of Cuntz--Krieger type algebras.

\begin{lemma}  \label{lemmaInvarianceGauge}
Let $J$ be a subset of $\A$.
Then the gauge actions exist on $(\F/\I)/\Sigma(J)$, so (A) is satisfied for the same $H$.
One has $\bal(q_J(x)) = \bal(x)$ for all words $x \in W$ with $q_J(x) \neq 0$.
If $\pi$ is a representation of $\F/\I$, $X$ a linear subspace of $\A$ and $J := \ker(\pi|_X)$
then the representation $\tilde \pi$ induced by $\pi$ by dividing out $J$ is injective on $q_J(X)$ ($\pi = \tilde \pi q_J$).
\end{lemma}

\begin{proof}
It is well known that $\A$ is the fixed point algebra of the gauge action $t$.
Hence, $t_\lambda(j) = j$ for $j \in J$ and $\lambda \in H$ since $J \subseteq \A =\lin(W_0)$.
Since an $x \in \Sigma(J)$ allows a representation $x=\sum_i \alpha_i a_i j_i b_i$ for scalars $\alpha_i \in \C$,
(possibly empty) words $a_i , b_i \in W$, and elements $j_i \in J$, this shows that
$t_\lambda(\Sigma(J)) \subseteq \Sigma(J)$ ($\lambda \in H$).
Hence the gauge actions exist on $(\F/\I)/\Sigma(J)$.
For the last claim, if $\tilde \pi(q_J(x)) = 0$ for $x \in X$, then $\pi(x) = 0$, then $x \in \ker(\pi|_X)$,
then $x \in J$, then $q_J(x) = 0$, showing that $\tilde \pi$ is injective on $q_J(X)$.
\end{proof}

\begin{definition} \label{DefCancelableFirst}
{\rm
An ideal $I \in \Sigma_\A$ is called cancelling if
$\F/\I$ divided by $I$ satisfies property (C').
}
\end{definition}

The proof of the next theorem will reveal that $I$ is cancelling if and only if $\F/\I$ divided by $I$ is a Cuntz--Krieger type $*$-algebra.
Write $\Omega_\A \subseteq \Sigma_\A$ for the family of all cancelling ideals.

\begin{theorem} \label{idealspropertyC}
We have $\Phi_\A(\Omega_\A)= \set{ D \in \cali}{D \cap \A \in \Omega_\A}$.
\end{theorem}

\begin{proof}
Define $\calj=\set{ D \in \cali}{D \cap \A \in \Omega_\A}$.
To prove $\Phi_A(\Omega_\A) \subseteq \calj$, consider an element $I \in \Omega_\A$,
and note that $\Phi_\A^{-1} (\Phi_\A(I)) = I = \Phi_\A(I) \cap \A \in \Omega_\A$
by Theorem \ref{injectlattice}.
Hence $\Phi_\A(I) \in \calj$.

To prove $\calj \subseteq \Phi_A(\Omega_\A)$ consider an element
$D \in \calj$.
Define $J= \Sigma(D\cap \A)$.
Write
$\pi:\F/\I
\longrightarrow C^*(\F/\I)/D$ for the canonical quotient map.
Write $C^*(J)$ for the norm closure of $J$ in $C^*(\F/\I)$.
Since $C^*(J) \subseteq D$, $\pi$ induces a homomorphism $\tilde \pi: (\F/\I)/J \longrightarrow C^*(\F/\I)/D$.
There is also a canonical homomorphism $\sigma: (\F/\I) /J \longrightarrow C^*(\F/\I)/ C^*(J)$.
Hence, by introducing a further quotient map $\lambda$,
we obtain a commutative diagram
$$
\begin{xy}
\xymatrix{
(\F/\I)/ J \ar[r]^{\tilde \pi} \ar[dr]_{\sigma} & C^*(\F/\I)/D  \\
  & C^*(\F/\I) / C^*(J) \ar[u]_\lambda
}
\end{xy}
$$

Since $D \cap \A = \ker(\pi|_\A)$, by Lemma \ref{lemmaInvarianceGauge} the algebra $(\F/\I)/J$ is invariant under the gauge
actions and $\tilde \pi$ is injective on $q_J(\A)$, which is the new core ``$\A$" for the algebra $(\F/\I)/J$
since $\bal(q_J(x))=\bal(x)$.
So $(\F/\I)/J$ is an algebra which satisfies (A) and (B), and there exists an $\A$-faithful $C^*$-representation $\tilde \pi$.
Since $J$ is generated by the cancelling ideal $D \cap \A \in \Omega_\A$, by Definition \ref{DefCancelableFirst} $(\F/\I)/J$ satisfies also (C') and so is a Cuntz--Krieger
$*$-algebra.

Hence, by Theorem \ref{CKuniqueness} the images of $\tilde \pi$ and $\sigma$ are canonically isomorphic, and so
$\lambda$ is proved to be an isomorphism.
By the definition of $\lambda$ this implies $C^*(J) = D$.
Since $D \in \calj$, $D \cap \A \in \Omega_\A$, and so
$D = C^*(J)= \Phi_\A(D\cap\A) \in \Phi_\A(\Omega_\A)$ as we wanted to show.
\end{proof}

\begin{corollary} \label{corollaryIdealsIso}
If all ideals in $\Sigma_\A$ are cancelling then $\Phi_\A$ is a lattice isomorphism.
\end{corollary}

\begin{proof}
Since all ideals in $\Sigma_\A$ are cancelling, $\Omega_\A = \Sigma_\A$.
By Theorem \ref{idealspropertyC}, $\Phi_\A$ is surjective.
By Theorem \ref{injectlattice} and Lemma \ref{lemmaLatticaA}, $\Phi_\A$ is an injective lattice homomorphism.
\end{proof}

We aim to generalise the last theorem by allowing $\A$ to be a smaller algebra $\B$.
The next definition will become clear in Corollary \ref{corollarySemigraphIsomorphPhiB} below.

\begin{definition} \label{DefCancelable}
{\rm
An ideal $I \in \Sigma_\B$ is called $\B$-cancelling if $X:=(\F/\I)/\Sigma(I)$ satisfies property (C'),
and every arbitrarily given $C^*$-representation of $X$ is injective on $q_I(\A)$ if and only if it is injective on $q_I(\B)$.
}
\end{definition}

Note that cancelling is the same as $\A$-cancelling.
Write $\Omega_\B \subseteq \Sigma_\B$ for the family of $\B$-cancelling ideals.
The next theorem and corollary generalise the last ones.

\begin{theorem} \label{idealspropertyCnumber2}
We have $\Phi_\B(\Omega_\B)= \set{ D \in \cali}{D \cap \B \in \Omega_\B}$.
\end{theorem}

\begin{proof}
This is proved exactly like Theorem \ref{idealspropertyC}.
One just replaces $\A$ by $\B$ and $\Omega_\A$ by
$\Omega_\B$ everywhere.
\end{proof}

\begin{corollary} \label{corollaryIdealsIso2}
If all ideals in $\Sigma_\B$ are $\B$-cancelling then $\Phi_\B$ is a bijection.
\end{corollary}

\begin{proof}
Since all ideals in $\Sigma_\B$ are $\B$-cancelling, $\Omega_\B = \Sigma_\B$.
By Theorem \ref{idealspropertyCnumber2} $\Phi_\B$ is surjective and
by Theorem \ref{injectlattice} $\Phi_\B$ is injective.
\end{proof}

We shall now apply the last corollary to cancelling higher rank
semigraph algebras \cite{burgiSemigraphI}.

\begin{corollary}  \label{corollarySemigraphIsomorphPhiB}
Let $\F/\I$ be a cancelling semigraph algebra (see \cite[Definitions 5.1 and 7.2]{burgiSemigraphI}),
and $\B$ the $*$-subalgebra of $\A$ generated by the standard projections (see \cite[Definition 5.14]{burgiSemigraphI}).
Then every quotient of $\F/\I$ by an ideal in $\Sigma_\B$ is a semigraph algebra by \cite[Lemma 8.1]{burgiSemigraphI}.
Now if every such quotient is cancelling (as a semigraph algebra), then $\Phi_\B$ is a bijection.
\end{corollary}

\begin{proof}
A $C^*$-representation of a cancelling semigraph algebra is injective on $\A$ if and only it is injective on $\B$
by \cite[Corollary 6.4]{burgiSemigraphI}.
If $I$ is an ideal in $\Sigma_\B$, then the image of $q_I$ is a semigraph algebra by \cite[Lemma 8.1]{burgiSemigraphI}.
The set of standard projections (see \cite[Definition 5.14]{burgiSemigraphI}) in the semigraph algebra $q_I(\F/\I)$ are the image of the standard projections
in $\F/\I$; so $q_I(\B)$ is the $*$-algebra generated by the standard projections in $q_I(\F/\I)$.
Note also that $q_I(\A)$ is the core, or the ``$\A$", of $q_I(\F/\I)$. Hence by \cite[Corollary 6.4]{burgiSemigraphI},
a $C^*$-representation of $q_I(\F/\I)$ is injective on $q_I(\A)$ if and only if it is injective on $q_I(\B)$.
So if we assume that $q_I(\F/\I)$ is cancelling (as a semigraph algebra), then it is a Cuntz--Krieger type $*$-algebra,
and so satisfies (C'), and by Definition \ref{DefCancelable} $I$ is $\B$-cancelling.

So if we assume that $q_I(\F/\I)$ is cancelling for every $I \in \Sigma_\B$, then $\Sigma_\B$ conists of $\B$-cancelling ideals only,
and so $\Sigma_\B = \Omega_\B$. The claim follows thus by Corollary \ref{corollaryIdealsIso2}.
\end{proof}

\begin{corollary}   \label{corollarySemigraphIsomorphPhiB2}
If every quotient of a cancelling semigraph algebra $\F/\I$ by an ideal in $\Sigma_\A$ is cancelling (as a semigraph algebra), then
$\Phi_\A$ is a lattice isomorphism.
\end{corollary}

\begin{proof}
One repeats the last three sentences of the proof of Corollary \ref{corollarySemigraphIsomorphPhiB} and replaces $\B$ by $\A$ everwhere.
\end{proof}

\section{Crossed Product Representation and Nuclearity} \label{section5}

By using the Cuntz--Krieger uniqueness theorem, Theorem \ref{CKuniqueness}, we can extend
each gauge action $t_\lambda \in \aut(\F/\I)$ to a gauge actions $\theta_\lambda \in \aut(C^*(\F/\I))$
($\lambda \in H$).
We may thus apply Takai's duality theorem \cite{0295.46088} and obtain the following result.

\begin{theorem} \label{takaidualityapplication}
By Takai's duality theorem we have
$$C^*(\F/\I) \otimes \calk(L^2(\calh)) \cong C^*(\F/\I) \rtimes_\theta H \rtimes_{\widehat \theta} \widehat H.$$
Moreover, $C^*(\F/\I) \rtimes_\theta H$
is the norm closure of a locally matricial algebra.
Hence $C^*(\F/\I)$ is nuclear.
\end{theorem}

\begin{proof}
The nuclearity is concluded from the observation that $C^*(\F/\I)$ is then evidently the corner of a crossed product of a (possibly non-separable) AF-algebra by an abelian group.

We assume that $\zeta$ is injective (Lemma \ref{lemmaInjectZeta}).
{\em Step 1.}
In the first step we follow the idea in \cite[Lemma 3.1]{raeburnszymanski}.
We denote the crossed product $C^*(\F/\I) \rtimes_\theta H$ by $A$.
Let $\calm(A)$ be the multiplier algebra of $A$.
Let $(U_\lambda)_{\lambda \in H} \subseteq \calm(A)$
be the unitaries inducing the actions $(\theta_\lambda)_{\lambda \in H}$.
Let
$$\chi(F) := \int_H F(\lambda) U_\lambda d \lambda \qquad \forall F \in \widehat H,$$
where we integrate in $\calm(A)$, and
where $d \lambda$ denotes the normalized Haar measure on $H$.
It is easy to see that
$(\chi(F))_{F \in \widehat H}$ forms a family of mutually orthogonal projections in $\calm(A)$.

Recall that $\bal(a)_\lambda a = \lambda_a a  = \theta_\lambda(a)$ for $a \in \cala$ and $\lambda \in H$, and we write the group operation of $\hat H$
additively.
Notice that
\begin{equation} \label{seq1}
\chi(F) a = a \chi(F+\bal(a)) \qquad \forall a \in \cala \, \forall F \in \widehat H.
\end{equation}
Notice that $a \chi(F) \in A$ for all $a \in \cala$ and $F \in \widehat H$.
By an application of the Stone-Weierstrass theorem the linear span of $\widehat H$ is dense in $L^1(H)$.
Hence $A$ is the norm closure of
$$B:= \lin \set{ \chi(F) x }{ x \in W, \, F \in \widehat H}.$$

{\em Step 2.}
It remains to show that $B$ is locally matricial.
Consider a finite subset
$$\Gamma=\{ \chi(F_{1}) x_1, \chi(F_{2}) x_2, \ldots, \chi(F_{n}) x_n\}$$
for some fixed nonzero $x_1,\ldots,x_n \in W$
and $F_1,\ldots,F_n \in \widehat H$.
By enlarging $\Gamma$, if necessary, we can assume that $\Gamma$ is self-adjoint (possible by identity
(\ref{seq1})).

Let $\omega$ be the set of nonzero words in the alphabet $\Gamma$.
By identity (\ref{seq1}) each $y \in \omega$ has a representation
$$y = \chi(F_{j_1}) x_{j_1} \chi(F_{j_2}) x_{j_2}\ldots \chi(F_{j_m}) x_{j_m}
= \chi(F_{j_{1}}) x_{j_1} x_{j_2}\ldots x_{j_m}$$ for some $1 \le
j_1,\ldots,j_m \le n$.
Since $y \neq 0$, we necessarily have
$$F_{j_{k+1}} = F_{j_k} + \bal(x_{j_k}) \qquad \forall k=1,\ldots,m-1.$$
Let
\begin{eqnarray*}
K &=& \{\, x_{j_1} x_{j_2} \ldots x_{j_m} \in \F/\I\, | \, m \ge 1, \, 1 \le j_1,\ldots,j_{m+1} \le n, \\
&& \qquad
F_{j_{k+1}} = F_{j_k} + \bal(x_{j_k}) \, \,\,\forall k = 1,\ldots,m\, \}.
\end{eqnarray*}
Notice that
$$\omega \subseteq \Gamma \cup \{ \chi(F_1), \ldots ,\chi(F_n) \} K \Gamma$$
(products in $A$).
Thus, if we can show that $K$ lies in some finite dimensional
space $\calm_n$ then
$\lin(\omega) = \alg^*(\Gamma)$ is a subspace of the finite dimensional
space
$$\lin(\Gamma \cup \{ \chi(F_1), \ldots ,\chi(F_n) \} \calm_n \Gamma),$$
and we are done.

We shall construct $\calm_n$ by induction.
Let $\gamma \subseteq \{1,\ldots,n\}$ and
$$L_{\gamma}:= \set{ x_{j_1} x_{j_2} \ldots x_{j_m}  \in K}
{ \{F_{j_1}, F_{j_2}, \ldots , F_{j_{m+1}}\} \subseteq \set{F_i}{i \in \gamma} }.$$

If $|\gamma|= 1$ then all $x_{j_k}$ of $x_{j_1} x_{j_2} \ldots x_{j_m} \in L_{\gamma}$ are zero-balanced.
Let $\calm_1 \subseteq \A$ be a finite dimensional $*$-algebra containing
$\set{ x_i \in \A}{ 1 \le i \le n, \, \bal(x_i) = 0}$.
Then it is clear that $L_{\gamma} \subseteq \calm_1$.

By induction hypothesis on $N=1,\ldots,n-1$ we assume that
there exists a finite dimensional vector space $\calm_N$, such that
$L_\gamma \subseteq \calm_N$
for all $\gamma \subseteq \{1,\ldots,n\}$ with $|\gamma| = N$.

Let $\delta \subseteq \{1,\ldots,n\}$ with $|\delta| = N+1$.
Let $x= x_{j_1} x_{j_2} \ldots x_{j_m} \in L_\delta$.
Let
$$\set{1 \le i \le m+1}{F_{j_i} = F_{j_1}} =: \{1=i_1 \le \ldots \le i_M \le m+1\}.$$
For $k=1,\ldots,M-1$ let
$$y_k = \prod_{t = i_k}^{i_{k+1}-1} x_{j_t} .$$
Since $y_k$ is a partial word of the word $x=x_{j_1} x_{j_2} \ldots x_{j_m}$ which lives in $K$,
we get
$$\bal(y_k) = \sum_{t=i_k}^{i_{k+1} -1} \bal(x_{j_{t}}) = \sum_{t=i_k}^{i_{k+1} -1} F_{j_{{t}+1}} - F_{j_{t}}
= F_{j_{i_{k+1}}} - F_{j_{i_k}} = F_{j_1} - F_{j_1} = 0.$$
Hence $y_k$ is zero-balanced and lives in $\A$.
We have
$$x = y_1 y_2 \ldots y_{M-1} x_{j_{i_M}}  x_{j_{i_M +1}} \ldots x_{j_m}.$$
Notice that  for all $k=1,\ldots,M$, both the `middle term' of
$y_k$, i.e.
$$x_{j_{i_k +1}} x_{j_{i_k +2}} \ldots x_{j_{i_{k+1} -2}},$$
and
the `end term' of $x$, i.e. $x_{j_{i_M +1}} \ldots x_{j_m}$, lie in
$L_{\delta \backslash \{j_1\}} \subseteq \calm_N$ (the inclusion is by induction hypothesis).
Thus $y_1,\ldots,y_{M-1}$ lie in the finite dimensional vector space
$$Y= \Big(\sum_{s=1}^n \C x_{s} + \sum_{s,t =1}^n \C x_{s} x_t+ \sum_{s,t =1}^n x_{s} \calm_N x_{t} \Big) \cap \A.$$
Hence $Z=\alg^*(Y)$ is a finite dimensional vector space since $Y\subseteq \A$.
Thus $y_1 \ldots y_{M-1} \in Z$, and
$x$ lies in the finite dimensional vector space
$$\calm_{N+1} =Z + \sum_{s=1}^n Z x_{s} + \sum_{s=1}^n Z x_s \calm_N.$$
Notice that the choice of $\calm_{N+1}$ is independent of $\delta$ and $x \in L_\delta$.
This completes the induction.
If $N+1= n$ then the proof is complete since then $K= L_{\{1,\ldots,n\}} \subseteq \calm_n$.
\end{proof}

{\bf Acknowledgement.} We would like to express our gratitude for
the support we received while working on this project at the
Universities of M\"unster and Rome ``Tor Vergata''; from the
Operator Algebras groups at the respective universities and the EU
IHP Research Training Network - Quantum Spaces and Noncommutative
Geometry.
The first named author thanks Joachim Cuntz for his invitation and hospitality
in M\"unster.

\bibliographystyle{plain}
\bibliography{references}

\begin{thebibliography}{10}

\bibitem{bratteli}
O.~Bratteli.
\newblock {Inductive limits of finite dimensional C$^*$-algebras.}
\newblock {\em Trans. Am. Math. Soc.}, 171:195--234, 1972.

\bibitem{burgiSemigraphI}
B.~Burgstaller.
\newblock {A Cuntz--Krieger uniqueness theorem for semigraph $C^*$-algebras}.
\newblock preprint. arXiv:1111.4166.

\bibitem{burgiCancelling}
B.~Burgstaller.
\newblock {Representations of crossed products by cancelling actions and
  applications.}
\newblock {\em Houston J. Math. to appear.}

\bibitem{burgi2}
B.~Burgstaller.
\newblock {The uniqueness of Cuntz--Krieger type algebras}.
\newblock {\em J. reine angew. Math.}, 594:207--236, 2006.

\bibitem{burgiHREL}
B.~Burgstaller.
\newblock {A class of higher rank Exel--Laca algebras}.
\newblock {\em Acta Sci. Math.}, 73:209--235, 2007.

\bibitem{burgiNoteUniqueness}
B.~Burgstaller.
\newblock {Notes on Cuntz--Krieger uniqueness theorems and $C^*$-algebras of
  labelled graphs}.
\newblock {\em Quaest. Math.}, 32:229--240, 2009.

\bibitem{cuntzkrieger}
J.~Cuntz and W.~Krieger.
\newblock {A class of $C^{*}$-algebras and topological Markov chains}.
\newblock {\em Invent. Math.}, 56:251--268, 1980.

\bibitem{kumjianpask}
A.~Kumjian and D.~Pask.
\newblock {Higher rank graph $C^*$-algebras}.
\newblock {\em New York J. Math.}, 6:1--20, 2000.

\bibitem{raeburnsimsyeendFinitelyAligned}
I.~Raeburn, A.~Sims, and T.~Yeend.
\newblock {The $C^*$-algebras of finitely aligned higher-rank graphs}.
\newblock {\em J. Funct. Anal.}, 213:206--240, 2004.

\bibitem{raeburnszymanski}
I.~Raeburn and W.~Szyma\'nski.
\newblock {Cuntz--Krieger algebras of infinite graphs and matrices}.
\newblock {\em Trans. Amer. Math. Soc.}, 356:39--59, 2004.

\bibitem{robertsonsteger2}
G.~Robertson and T.~Steger.
\newblock {Affine buildings, tiling systems and higher rank Cuntz--Krieger
  algebras}.
\newblock {\em J. reine angew. Math.}, 513:115--144, 1999.

\bibitem{0295.46088}
H.~Takai.
\newblock {On a duality for crossed products of C$^*$-algebras.}
\newblock {\em J. Funct. Anal.}, 19:25--39, 1975.

\bibitem{tomforde}
M.~Tomforde.
\newblock {A unified approach to Exel-Laca algebras and $C^*$-algebras
  associated to graphs}.
\newblock {\em J. Oper. Theory}, 50(2):345--368, 2003.

\end{thebibliography}

\end{document}